\documentclass[12pt]{amsart}
\usepackage{color}
\usepackage[all]{xy}
\usepackage{amssymb,amsmath}

\usepackage{mathrsfs}
\usepackage{eucal}
\usepackage{enumerate}

\usepackage{setspace}

\usepackage{hyperref}
\hypersetup{
	colorlinks=true,
	linkcolor=blue,
	filecolor=blue,
	citecolor = blue,
	urlcolor=blue,
}

\date{April 30, 2026}

\calclayout
\newtheorem{dummy}{anything}[section]
\newtheorem{theorem}[dummy]{Theorem}

\newtheorem{lemma}[dummy]{Lemma}
\newtheorem{proposition}[dummy]{Proposition}
\newtheorem{corollary}[dummy]{Corollary}

\theoremstyle{definition}%%Change Theoremstyle
\newtheorem{definition}[dummy]{Definition}
  \newtheorem{example}[dummy]{Example}
  
  \newtheorem{remark}[dummy]{Remark}

    \newtheorem*{question}{Question}

\newcommand{\cD}{\mathcal D}

\newcommand{\cF}{\mathcal F}
\newcommand{\cG}{\mathcal G}
\newcommand{\cH}{\mathcal H}

\newcommand{\cM}{\mathcal M}

\newcommand{\cO}{\Or}
\newcommand{\cP}{\mathscr P}

\newcommand{\cS}{\mathcal S}

\newcommand{\bC}{\mathbf C}
\newcommand{\bD}{\mathbf D}

\newcommand{\bZ}{\mathbb Z}

\newcommand{\bP}{\mathbf P}

\newcommand{\bbC}{\mathbb C}
\newcommand{\bbF}{\mathbb F}

\newcommand{\bbQ}{\mathbb Q}
\newcommand{\bbR}{\mathbb R}

\newcommand{\bbZ}{\mathbb Z}

 \DeclareMathOperator{\Ind}{Ind}
\DeclareMathOperator{\Res}{Res}
\DeclareMathOperator{\Iso}{\mathfrak{Iso}}
\DeclareMathOperator{\Mor}{Mor}
 \DeclareMathOperator{\Map}{Map}
\DeclareMathOperator{\ind}{Ind}
 \DeclareMathOperator{\Dim}{Dim}
\DeclareMathOperator{\HomDim}{hDim}
\DeclareMathOperator{\hdim}{hdim}
  \DeclareMathOperator{\rk}{rank}

\DeclareMathOperator{\Hom}{Hom}
\DeclareMathOperator{\End}{End}
 
\DeclareMathOperator{\Supp}{Supp}
 \DeclareMathOperator{\Inc}{Inc}
 \DeclareMathOperator{\res}{Res}
 \DeclareMathOperator{\Mod}{Mod}
 \DeclareMathOperator{\Def}{Def} 
\newcommand{\RMod}{R\textrm{-Mod}}
\DeclareMathOperator{\Inf}{Inf}
\newcommand\uRb[1]{R[{#1}\hphantom{{\hskip-4pt}l}^{\textbf{?}\,}]}
\newcommand\oG{\overline{G}}
\newcommand\oK{\overline{K}}
 \DeclareMathOperator{\Ext}{Ext}

\newcommand{\cy}[1]{\bZ/{#1}}
\newcommand{\la}{\langle}
\newcommand{\ra}{\rangle}
\newcommand{\vv}{\, | \,}
\newcommand{\bd}{\partial}
\newcommand{\id}{\mathrm{id}}
\newcommand\Fp{\bbF_p}
\def\bZp{\bZ_{(p)}}

\def\G{\varGamma}

\newcommand{\disjointunion}{\bigsqcup}

\newcommand\uHom{\Hom_{\RG_\textbf{?}}}
\newcommand{\uExt}{\Ext_{\RG_\textbf{?}}}

\DeclareMathOperator{\Or}{Or}
\newcommand{\Sub}{{\cS}(G)}
\newcommand\OrG{\Or _{\cF}G}
\newcommand\OrH{\Or_{\cF}H}

\newcommand\RG{R\G}
\newcommand\ZG{\bZ\G}
\newcommand\un{\underline}
\newcommand{\up}{^{(p)}}

\newcommand\uR[1]{R[{#1}^{\, \textbf{?}\,}]}
\newcommand\uC[1]{\CC({#1}^{\textbf{?}};R)}
\newcommand\uCZ[1]{\CC({#1}^{\textbf{?}};\bZ)}
\newcommand\CC{\bC}
\newcommand\DD{\bD}

\newcommand\PP{\bP}

\newcommand\uH[1]{H_*({#1}^{\textbf{?}};R)}
\DeclareMathOperator{\Qd}{Qd}

\newcommand{\PG}{\cS_G}
\DeclareMathOperator{\Fix}{Fix}

\DeclareMathOperator{\Pic}{Pic}

\def\bZp{\bbZ_{(p)}}

\begin{document}

\title{Equivariant CW-complexes homotopy equivalent to spheres: a survey}
\author{Ian Hambleton}
\author{Erg\"un Yal\c c\i n}

\address{Department of Mathematics, McMaster University,
Hamilton, Ontario L8S 4K1, Canada}

\email{hambleton@mcmaster.ca }

\address{Department of Mathematics, Bilkent University,
06800 Bilkent, Ankara, Turkey}

\email{yalcine@fen.bilkent.edu.tr }

\thanks{Research partially supported by NSERC Discovery Grant A4000.}

\begin{abstract}
This is a survey about finite group actions on CW-complexes and related topics, primarily based on our joint work. The main applications are to finite $G$-CW-complexes which are homotopy equivalent to spheres.  We have tried to give a fairly short overview of the extensive literature in this area, and we apologize in advance for our oversights and omissions.  
\end{abstract}

\maketitle

\section{ Introduction}
\label{sect:introduction} 
Let $G$ be a finite group. The unit spheres $S(V)$ in  finite-dimensional orthogonal representations of $G$ provide the basic examples of smooth linear $G$-actions on spheres. These linear  actions satisfy a number of special constraints on the dimensions of fixed sets and the structure of the isotropy subgroups, arising from character theory.  However, such constraints do not hold in general for smooth $G$-actions on spheres, unless $G$ has prime power order.
In a series of papers \cite{Hambleton:2010,Hambleton:2013,Hambleton:2014,Hambleton:2016,Hambleton:2017} we constructed new families of equivariant $G$-CW-complexes homotopy equivalent to spheres, with prescribed isotropy, as potential models for smooth \emph{non-linear} finite group actions on spheres.

%new examples of smooth \emph{non-linear} finite group actions on spheres, with prescribed isotropy.  

In this series we studied group actions on spheres in the setting of 
\emph{geometric $G$-homotopy representations},
introduced by tom Dieck \cite{Dieck:1982,Dieck:1982a,Dieck:1986}. These are finite (or more generally finite-dimensional) $G$-CW-complexes $X$ satisfying the property that for each $H\leq G$, the fixed point set $X^H$ is homotopy equivalent to a sphere $S^{n(H)}$ where $n(H)=\dim X^H$.  
We introduced algebraic homotopy representations as suitable chain complexes over the orbit category and proved a realization theorem for these algebraic models. 

In the study of finite group actions, the $p$-Sylow subgroups play a central role due to the constraints imposed by P.~A.~Smith theory on the homology of the fixed sets (see \cite{Smith:1938,Smith:1939,Smith:1941,Smith:1945}). We will focus on the interactions between the family of isotropy subgroups of the action and the dimensions of the fixed subspaces  for finite group actions on spheres (see the survey \cite{Hambleton:2015} for information about free actions on spheres).

\tableofcontents

\section{ Some motivating questions}\label{sec:two}
Actions of finite groups on spheres can be studied in various different geometrical settings.  The fundamental examples come from the unit spheres $S(V)$ in a real orthogonal  $G$-representation $V$, and character theory reveals intricate relations between the dimensions of the fixed subspheres $S(V)^H$, for subgroups $H \leq G$, and the structure of the isotopy subgroups $\{G_x \vv x \in S(V)\}$.
 One goal is to  better understand  the constraints on  these basic invariants,  in order to construct new smooth \emph{non-linear} finite group actions on spheres. 
  
A useful way to measure the complexity of the isotropy is the \emph{rank}.
We say that $G$ has \emph{rank $k$} if it contains a subgroup isomorphic to $(\bbZ/p)^k$, for some prime $p$, but no subgroup $(\bbZ/p)^{k+1}$, for any prime $p$. 
After free actions where the isotropy is trivial, the next simplest case is the following:

\begin{question}\label{ques:mainquestion} For which finite groups  $G$, does there exist a finite $G$-CW-complex $X\simeq
S^n$ with all isotropy subgroups of rank $k$ ~?
\end{question}

By P.~A.~Smith theory, the rank $k$  assumption on the isotropy subgroups implies that $G$ must have $\rk(G)\leq k+1$ (see \cite[Corollary 6.3]{Hambleton:2016}). For example, since every rank one finite group can act freely on a finite
complex homotopy equivalent to a sphere (Swan \cite{Swan:1960}), we could restrict our attention to groups  with rank $\geq 2$. Here are three  natural settings for the study of finite group actions on spheres:

\begin{enumerate}[(A)]
\item  smooth $G$-actions
  on closed manifolds homotopy equivalent to spheres; 
\item finite $G$-homotopy representations  (see tom Dieck 
\cite[II.10.1]{Dieck:1987});
\item finite $G$-CW-complexes $X \simeq S^n$.
\end{enumerate}

In contrast to  $G$-representation spheres $S(V)$, the non-linear smooth $G$-actions on a smooth manifold $M \simeq S^n$ exhibit more flexibility. For example, in the linear case,  the fixed sets $S(V)^H$ are always linear subspheres. For smooth actions, the fixed sets are smoothly embedded submanifolds but may not even be integral homology spheres \cite{Petrie:1979}. 
Well-known general constraints on smooth actions arise from P.~A.~Smith theory: 
\begin{theorem}[P.~A.~Smith \cite{Smith:1938}]
Let $p$ be a prime and $G$ be a finite $p$-group. If $X$ is a finite-dimensional $G$-CW-complex which has the mod-$p$ homology of an $n$-sphere, then the fixed point set $X^G$ is either empty or has the mod-$p$ homology of an $r$-sphere for some $0\leq r \leq n$. If $p$ is odd, then $n-r$ is even.
\end{theorem}
In particular, if $G$ acts smoothly on a smooth manifold $M \simeq S^n$, and $H$ is a subgroup of $p$-power order, for some prime $p$,  then $M^H$ is an $\Fp$-homology sphere.  In addition, even if the fixed sets are diffeomorphic to spheres, they may be knotted or linked as embedded subspheres in $M$ (see \cite{Dieck:1985b}, \cite{Davis:1988}).  One can also consider topological $G$-actions, usually with the assumption of local linearity, otherwise the fixed sets may not be locally flat  submanifolds. There is a large literature concerning smooth actions on manifolds,  as in setting (A), in particular on their  fixed sets and isotropy structures, but these topics deserve a full survey on their own.  

For example,  one  problem (first proposed by Montgomery and Samelson) that has attracted a lot of attention is the existence of smooth actions on spheres with exactly one fixed point. The first step towards a solution of this problem was 
the striking work of L.~Jones \cite{Jones:1971,Jones:1972}, establishing a converse to the  P.~A.~Smith theorem for the existence of actions on spheres and disks. The first one fixed point example was discovered by E.~Stein \cite{Stein:1977}, and then the work of R.~Oliver and T.~Petrie \cite{Oliver:1975,Oliver:1982}  opened the way to a systematic approach to constructing more examples via equivariant surgery (see \cite[Theorem A]{Petrie:1982} and \cite{Laitinen:1998}). 

In the setting (B) of $G$-homotopy representations, the objects of study are finite (or more generally finite-dimensional) $G$-CW-complexes $X$ satisfying the property that for each $H\leq G$, the fixed point set $X^H$ is homotopy equivalent to a sphere $S^{n(H)}$ where $n(H)=\dim X^H$. 
We could also consider a version of this setting where all isotropy subgroups of $X$ are $p$-subgroups of $G$, and each non-empty fixed point set $X^H$ for $H\leq G$ is a $\bZp$-homology $n(H)$-sphere where $n(H)=\dim X^H$.
 
The third setting (C) is the most flexible of all. Here we suppose that $X\simeq S^n$ is a finite $G$-CW-complex homotopy equivalent to a sphere, but do not require that $\dim X = n$. Moreover, we make no initial assumptions about the homology of the fixed sets $X^H$, although the conditions imposed by P.~A.~Smith theory with
 $\bbF_p$-coefficients still hold. In the setting (C), we will see that  $\dim X^H$ can be  (much) higher in general than its homological dimension, and this provides new obstructions to understanding our motivating question in setting (A) or (B).

 \section{ The foundational work of tom Dieck}\label{sec:three} 
 This section is a tribute and recollection of some of tom Dieck's many contributions to the theory of transformation groups, and specifically to the geometry of representations. 
 
 The study of orthogonal representations of finite groups (or more generally compact Lie groups) up to homotopy was pioneered by tom Dieck in the late 1970's (see \cite{Dieck:1978b,Dieck:1978,Dieck:1979b,Dieck:1979a}). The focus was on the pattern of dimensions for the fixed sets in representation spheres $S(V)$.  
 At the same period, the definition of the Burnside ring was extended to compact Lie groups and its role in investigated and followed up in joint work with Petrie \cite{Dieck:1975b,Dieck:1975a,Dieck:1977a,Dieck:1977,Dieck:1978a,Dieck:1979}.
 
 Important tools were introduced and developed,  such as equivariant homology theory and Mackey functors  \cite{Dieck:1972,Dieck:1973,Dieck:1975}, equivariant classifying spaces \cite{Dieck:1969,Dieck:1974}, and the stable equivariant  Picard group associated to the real representation ring $RO(G)$ of a compact Lie group \cite{Dieck:1978b,Dieck:1984b,Dieck:1985}. These concepts and methods have been important in the development of equivariant spectra (see 
  Lewis, May, Steinberger, and McClure \cite{Lewis:1986}).

 In the early 1980's tom Dieck and Petrie transformed the subject by introducing \emph{homotopy representations},  and embarked on a decade of research activity centered around this notion (see \cite{Dieck:1982b,Dieck:1982a,Dieck:1982,Dieck:1986}). The fundamental questions about dimensions and isotropy which appeared in representation theory now had analogues for $G$-CW-complexes. An excellent account of this whole area appears in tom Dieck's ICM article \cite{Dieck:1987a}. 
 
  Results in this more general setting then suggested geometric questions for smooth or topological group actions on spheres, such as possible linking numbers of disjoint fixed sets, and the application of tools from surgery theory \cite{Dieck:1988,Davis:1988,Dieck:1989b,Dieck:1989a,Dieck:1991a}. 
    
 We conclude this section by mentioning another aspect of the ideas presented by tom Dieck and Petrie in 
 \cite{Dieck:1982a}, which is the subject of current research. They defined the group $V(G)$ as the Grothendieck group of the monoid of equivalence classes of finite-dimensional $G$-homotopy representations, with the group operation given by join operation. They also introduced the concept of generalized homotopy representations:  a finite-dimensional $G$-CW-complex $X$ is a \emph{generalized homotopy representation} if the fixed point set $X^H$ is homotopy equivalent to $S^{n(H)}$, where $n(H)$ is not necessarily equal to $\dim X^H$. They defined a similar Grothendieck group $V'(G)$ for generalized homotopy representations and prove the following:
 
 \begin{theorem}[{tom Dieck and Petrie \cite{Dieck:1982a}}] If $G$ finite or a torus, the canonical map $V(G) \to  V'(G)$ is an isomorphism.
 \end{theorem}
 
 It is an open problem whether the same statement holds for arbitrary compact Lie groups  (see \cite[Conjecture 4.5]{Fausk:2001} and \cite{Knutsen:2023}). The group $V'(G)$ is isomorphic to the Picard group  $\Pic (hSp^G )$ of invertible objects in the stable homotopy category of $G$-spectra (see \cite[Proposition 4.6]{Fausk:2001}), and the group $V(G)$ fits into an exact sequence 
 \[ 
 0 \to \Pic (A(G)) \to V(G) \to D(G) \to 0 
 \]
 where $\Pic (A(G))$ denotes the Picard group of the Burnside ring $A(G)$, and $D(G)$ denote the group described by Bauer \cite{Bauer:1988}.
 As a consequence of the above theorem, one obtains a description of  the Picard group $\Pic (hSp^G )$ when $G$ is a finite group.

For subsequent work on equivariant $G$-CW-complexes, tom Dieck's 1987 book ``Transformation Groups''  \cite{Dieck:1987} provided an equivariant framework based on modules and chain complexes over the  \emph{orbit category} of 
$G$, which includes tools from algebraic topology, stable homotopy, the Burnside ring, Bredon homology and Mackey functors (and much more). 

With this preparation, the reader is equipped to absorb further extensions of the algebra, and the connections between transformation groups and algebraic K- theory presented in the influential LNM1408 by W.~L\"uck   \cite{Luck:1989}. In this comprehensive exposition, the author introduces invariants such as the equivariant finiteness obstruction, Whitehead torsion, and Reidemeister torsion taking values in algebraic K-groups, and presents geometric applications to isovariant and equivariant versions of the $s$-cobordism theorems for $G$-manifolds. For the present survey, one of the highlights is that  (oriented) $G$-homotopy representations with the same dimension function are classified up to $G$-homotopy equivalence by the reduced Reidemeister torsion invariant (see  \cite[Corollary 20.39]{Luck:1989}).

\section{ Chain complexes over the orbit category}\label{sec:four}

Let $G$ be a finite group and $\cF$ be a family of subgroups of $G$ which is closed under conjugations and taking subgroups. 
The orbit category $\OrG$ is defined as the category whose objects are orbits of type
$G/K$, with $K \in \cF$, and where the morphisms from $G/K$ to $G/L$
are given by $G$-maps:
$$ \Mor _{\OrG} (G/K , G/L ) = \Map _G (G/K, G/L ).$$
 The special case of $\OrG$ where $\cF$ consists of all subgroups of $G$ is just called the \emph{orbit category} and denoted $\Or(G)$.
 
The category $\G _G= \OrG$ is a small category, and we can consider the
module category over $\G _G$ in the following sense. Let $R$ be a commutative ring with 1.
A \emph{(right) $\RG _G$-module} $M$ is a contravariant functor from $\G _G$ to
the category of $R$-modules. We denote the $R$-module $M(G/K)$
simply by $M(K)$ and write $M(f)\colon  M(L) \to M(K)$  for a $G$-map $f\colon 
G/K \to G/L$. 

The category of $\RG _G$-modules is an abelian category, so the usual concepts of homological algebra, such as kernel, direct sum,
exactness, projective module, etc., exist for $\RG_G$-modules.
A sequence of $\RG _G$-modules $0 \to A \to B \to C \to 0$
is \emph{exact} if and only if $$ 0 \to A(K) \to B(K) \to C(K) \to 0$$ is
an exact sequence of $R$-modules for every $K \in \cF$. For an $\RG_G$-module $M$ the $R$-module $M(K)$ can also be considered as an $RW_G(K)$-module in an obvious way where $W_G(K) = N_G(K)/K$. We will follow the convention in \cite{Luck:1989} and  consider $M(K)$ as a right $RW_G(K)$-module. In particular, we will consider the sequence above as an exact sequence of right $RW_G(K)$-modules.

For each $H \in \cF$, let $F_H:=R[G/H ^?]$ denote the $R\G_G$-module with values $F_H (K)=R[(G/H)^K]$ for every $K \in \cF$, and where for every  $G$-map $f: G/L \to G/K$, the induced map $F_H (f): R[(G/H)^K]\to R[(G/H)^L]$ is defined in the obvious way. By the Yoneda lemma, there is an isomorphism $$\Hom _{R\G_G} (R[G/H^?], M)\cong M(H)$$ for every $R\G_G$-module $M$. From this it is easy to show that the module $R[G/H^?]$ is a projective module in the usual sense, for each $H \in  \cF$. An $R\G_G$-module is called \emph{free} if it is isomorphic to a direct sum of $R\G_G$-modules of the form $R[G/H^?]$. It can be shown that an $R\G_G$-module is projective if and only if it is a direct summand of a free module. The further details about the properties of modules over the orbit category can be found in \cite{Hambleton:2013} (see also L\"uck \cite[\S 9,\S 17]{Luck:1989}  and tom Dieck \cite[\S 10-11]{Dieck:1987}).   

We consider chain complexes $\CC$ of $\RG _G$-modules, with respect to a given family $\cF$. When we say a chain complex we always mean a non-negative complex, so $\CC _i=0$ for $i<0$. We call a chain complex $\CC$ \emph{projective} (resp.~\emph{free}) if for all $i\geq 0$,
the modules $\CC_i$ are projective (resp.~free). We say that a chain complex $\CC$ is \emph{finite} if $\CC_i = 0$ for all $i > n$, for some $n\geq 0$, and the chain modules $\CC_i$ are all finitely generated $\RG_G$-modules. We usually work under this assumption.

We define the \emph{support} of a chain complex $\CC$ over $R\G_G$ as the family of subgroups $$\Supp(\CC) = \{ H \in \cF \vv \CC(H) \neq 0\}.$$

It is sometimes convenient to vary the family of subgroups. 

\begin{definition}
If $\cF \subset \cG$ are two families, the orbit category $\G _{G, \cF}=\Or_\cF G$ is a full-subcategory of $\Gamma _{G, \cG}=\Or_\cG G$. If $M$ is a module over $R\G_{G, \cF}$, then we define
$\Inc_\cF^\cG(M)(H) =M(H)$, if $H \in \cF$, and zero  otherwise. 
Similarly, 
for a module $N$ over $R\G_{G, \cG}$,  define $\res_\cF^\cG(N)(H) = N(H)$, for $H \in \cF$. We extend to maps and chain complexes similarly. Note that $\Supp(\Inc_\cF^\cG(\CC)) = \Supp(\CC)$, and $\Supp(\res_\cF^\cG(\DD)) = \Supp(\DD) \cap \cF$.  
\end{definition}

 Given a $G$-CW-complex $X$, 
there is an associated chain complex of $\RG_G$-modules over the family of all subgroups of $G$:
$$\uC{X}:   \quad  \cdots \to \uR{X_n}
\xrightarrow{\bd_{n}}  \uR{X_{n-1}} \rightarrow
\cdots\xrightarrow{\bd_{1}}  \uR{X_0} \to 0,$$ where $X_i$ denotes
the set of (oriented) $i$-dimensional cells in $X$ and $\uR{X_i}$ is the
$\RG_G$-module defined by $ \uR{X_i}(H)= R[X_i^H]$  for every $H \leq G$. We denote the
homology of this complex by $\uH{X}$. The chain complex $\CC (X^H; R)$ is actually defined for all subgroups $H \leq G$, but for a given family of subgroups $\cF$, we can restrict its values from $\Or(G)$ to the full sub-category $\OrG$.  

The smallest family containing all the isotropy subgroups  $\{G_x \vv x \in X\}$ is 
$$\Iso(X) = \{ H \leq G \vv X^H \neq \emptyset\}$$
and this motivates our notion of support for algebraic chain complexes.
In particular, we have 
$$ \Supp(\res_\cF(\uC{X})) = \cF \cap \Iso(X). $$  
If the family $\cF$ includes all the isotropy subgroups of $X$, then the complex $\uC{X}$ is a chain complex of free $\RG_G$-modules, hence projective $\RG_G$-modules, but otherwise the chain modules $\uR{X_n}$ may not be projective over $\RG _G$.

\section{ Dimension Functions}\label{sec:five}
Given a finite $G$-CW-complex $X$,  there is a
\emph{dimension function} $$\Dim X \colon  \Sub  \to \bZ,$$  given by $(\Dim X )(H)=\dim X^H$ for all $H \in \Sub$ where $\Sub$ denotes the set of all subgroups of $G$. By convention,  we set $\dim \emptyset = -1$ for the dimension of the empty set.
In a similar way, we define the following.

\begin{definition}\label{defn:Dim} The (chain) \emph{dimension function} of a finite-dimensional chain complex $\CC$ over $\RG_G$ is defined as the function $\Dim \CC \colon  \Sub \to \bZ$ which has the value
 $$(\Dim \CC )(H)=\dim \CC (H)$$ for all 
 $H \in \cF$, 
 where the \emph{dimension} of a chain complex of $R$-modules is defined as the largest integer $d$ such $C_d\neq 0$ (hence the zero complex has dimension $-1$). If 
$H \notin \cF$, then we set $(\Dim \CC)(H)=-1$. 
\end{definition}

The dimension function $\Dim \CC \colon \Sub \to \bZ$ is  \emph{conjugation-invariant}, meaning that it takes the same value on subgroups conjugate in $G$.  The term \emph{super class function} is often used for such functions. 

\begin{definition}\label{defn:superclass}
 The \emph{support} of a conjugation-invariant function $\un{n} : \Sub \to \bbZ$ is defined as the set 
$$\Supp (\un{n})=\{ H \leq G \colon \un{n} (H) \neq -1 \}.$$
We say that a conjugation-invariant function $\un{n}\colon \Sub  \to \bZ$ is \emph{supported on $\cF$}, if  $\Supp(\un{n}) \subseteq \cF$. Note that $\Supp (\CC)\subseteq \cF$ is the support of the dimension function $\Dim\CC$ of a chain complex $\CC$ over $\RG_G$. 
\end{definition}
 
In a similar way, we can define the \emph{homological dimension function} of a chain complex $\CC$ of $\RG_G$-modules
as the function $\HomDim \CC \colon  \Sub \to \bZ$ where for each $H\in \cF$, the integer 
$$(\HomDim \CC ) (H) = \hdim \CC(H)$$
is defined as the largest integer $d$ such that  $H_d (\CC (H))\neq 0$. If $H \notin \cF$, then we set $\un{n}(H) = -1$, as before.

Let us write $(H) \leq (K)$ whenever $g^{-1} Hg \leq K$ for some $g\in G$. Here $(H)$ denotes the set of subgroups conjugate to $H$ in $G$. The notation $(H) < (K)$ means that $(H)\leq (K)$ but $(H)\neq (K)$.

\begin{definition}\label{def:monotone}
We call a function $\un{n}\colon \Sub \to  \bZ$ \emph{monotone} if it
satisfies the property that $\un{n}(H) \geqslant \un{n}(K)$ whenever
$(H) \leq (K)$. We say that a monotone function  $\un{n}$ is
\emph{strictly monotone} if $\un{n}(H) > \un{n}(K)$, whenever $(H)< (K)$.  \qed 
\end{definition} 

We have the following:

\begin{lemma}\label{lem:monotone} The (chain) 
dimension function of every finite-dimensional projective chain complex $\CC$ of $\RG_G$-modules is monotone.
\end{lemma}

\begin{proof} Let $(L)\leq (K)$. If $\un{n}(K)=-1$, then the inequality $\un{n}(L) \geq \un{n}(K)$ is clear. So assume $\un{n}(K)=n \neq -1$. Then $\CC_n (K) \neq 0$. By the decomposition theorem for projective $\RG_G$-modules \cite[Chap.~I, Theorem 11.18]{Dieck:1987}, every projective $\RG _G$-module  $P$ is of the form $P\cong \oplus _H E_H P_H$, where $H\in \cF$ and $P_H$ is a projective $N_G (H)/H$-module. Here the $R\G _G$-module $E_HP_H$ is defined by $$E_H P_H (?)=P_H \otimes _{RN_G(H)/H} R\Map_G(G/?, G/H).$$ Applying this decomposition theorem to $\CC _n$, we observe that $\CC_n$ must have a summand $E_H P_H$ with $(K)\leq (H)$. But then $\CC_n(L)\neq 0$, and hence $\un{n}(L) \geq \un{n}(K)$.  
\end{proof}

We are particularly interested in chain complexes which have the homology of a sphere when evaluated 
at every $K \in \cF$. To specify the restriction maps in dimension zero, we will consider chain complexes $\CC$ which are equipped with  an augmentation map $\varepsilon\colon \CC _0 \to \un{R}$ such that $\varepsilon \circ \bd _1=0$. Here $\un{R}$ denotes the constant functor, and we assume that $\varepsilon(H)$ is surjective for $H \in \Supp(\CC)$.
We often consider $\varepsilon$ as a chain map $\CC \to \un{R}$ by considering $\un{R}$ as a chain complex over $\RG_G$ which is concentrated at zero. We denote a chain complex with an augmentation as a pair $(\CC, \varepsilon)$.

By the \emph{reduced homology} of a complex $(\CC,\varepsilon)$, we always mean the homology of the augmented chain complex 
$$ \widetilde \CC = \{ \cdots \to \CC_n \xrightarrow{\bd_n} \cdots  \to \CC _2 \xrightarrow{\bd_2} \CC _1\xrightarrow{\bd _1}
 \CC _0\xrightarrow{\varepsilon} \un{R} \to 0\}$$ 
where $\un{R}$ is considered to be at dimension $-1$. Note that the complex $\widetilde  \CC$ is the $-1$ shift of the mapping cone of the chain map $\varepsilon \colon \CC \to \un{R}$.    

\begin{definition}\label{def:Rhomologysphere}
 Let $\un{n}$ be a conjugation-invariant function supported on $\cF$, and let $\CC$ be a chain complex over $\RG_G$ with respect to a family $\cF$ of subgroups.  
 \begin{enumerate}
\item  We say that $\CC$ is an \emph{$R$-homology $\un {n}$-sphere} if there is an augmentation map $\varepsilon\colon \CC \to \un{R}$ such that the reduced homology of $\CC (K)$ is  the same as that of an $\un{n}(K)$-sphere (with coefficients in $R$) for all $K \in \cF$.  
\item We say that $\CC$ is \emph{oriented} if the $W_G
(K)$-action  on the homology of $\CC (K)$ is trivial for all $K\in \cF$. 
\end{enumerate}
\end{definition}

Note that we do not assume that the dimension function is strictly monotone as in Definition II.10.1 in \cite{Dieck:1987}. 

%In transformation group theory, a $G$-CW-complex $X$ is called a \emph{homotopy representation} if $X^H$ is homotopy equivalent to the sphere $S^{\un{n}(H)}$ where $\un{n}(H)=\dim X^H$ for every $H\leq G$ (see tom Dieck  \cite[Section II.10]{tomDieck2}). We now introduce an algebraic analogue of this useful notion for chain complexes over the orbit category.  

\section{ Algebraic Homotopy Representations}\label{sec:six}
In order to put our questions in a more general setting,  we study an algebraic version of tom Dieck's geometric homotopy representations for $R$-module chain complexes over the orbit category $\G _G =\Or _{\cF} G$, with respect to a ring $R$ and a family $\cF$ of subgroups of $G$. We usually work with $R = \bZp$, for some prime $p$, or $R=\bbZ$. This theory was developed by L\"uck \cite[\S 9, \S 17]{Luck:1989} and tom Dieck \cite[\S 10-11]{Dieck:1987}.
 
The homological dimensions of the various fixed sets are encoded in a conjugation-invariant function  
$\un{n} \colon \Sub \to \bbZ$,  where $\Sub$ denotes the set of subgroups of $G$. The function $\un{n}$ is \emph{supported on the family $\cF$},  if  $\un{n}(H) = -1 $ for $H \notin \cF$ (see Definition \ref{defn:superclass}). 

In \cite[II.10]{Dieck:1987}, there is a list of properties that are satisfied by geometric homotopy representations. We will use algebraic versions of these properties to define an analogous notion for chain complexes.

\begin{definition}\label{def:algrep} Let $\CC$ be a finite 
projective  chain complex over $\RG_G$, which is an $R$-homology $\un{n}$-sphere.
We say $\CC$ is an \emph{algebraic homotopy representation} (over $R$) if 
\begin{enumerate}
\item The function $\un{n}$ is a monotone
function.
\item If $H,K \in \cF$ are such that $n=\un{n}(K)=\un{n}(H)$,
then for every $G$-map $f \colon  G/H \to G/K$ the induced map $\CC(f)\colon \CC (K) \to \CC (H)$ is an $R$-homology isomorphism.
\item Suppose $H, K, L\in \cF$ are such that $H \leq K,L$ and
let $M=\langle K, L \rangle$ be the subgroup of $G$ generated by $K$ and $L$. If
$n=\un{n}(H)=\un{n}(K)=\un{n}(L) >-1$, then $M \in \cF$ and
$n=\un{n}(M)$.  
\end{enumerate}
\end{definition}

 These are the internal homological conditions observed for representation spheres, and more generally, these conditions are all necessary for $\CC$ to be chain homotopy equivalent to a geometric homotopy representation (see \cite{Laitinen:1986} and
  \cite[Section 2]{Hambleton:2014}).  Note that conditions (ii) and (iii) of Definition \ref{def:algrep} are automatic if the dimension function $\un{n}$ is strictly monotone. Under condition (iii),  the isotropy family $\cF$ has an important maximality property.

\begin{proposition}\label{cor:max} Let $\un{n}$ be a conjugation-invariant function and let $\CC$ be a %
projective chain complex of $\RG_G$-modules, which is an $R$-homology $\un{n}$-sphere. If condition \textup{(iii)} holds, then for each $H \in \cF$, the set of subgroups $\cF_H = \{ K \in \cF \vv (H) \leq (K),\ \un{n}(K) = \un{n}(H)> -1\}$ has a unique maximal element, up to conjugation. 
\end{proposition}

\begin{proof} Clear by induction from the statement of condition (iii). 
\end{proof}

In the remainder of this section we will assume that $R$ is a principal ideal domain. 
 
\begin{proposition}
\label{prop:tightness conditions} Let $\un{n}$ be a conjugation-invariant function and $\CC$ be 
a finite projective  chain complex over $\RG_G$, which is an $R$-homology $\un{n}$-sphere. Assume that $R$ is a principal ideal domain. If the equality $\un{n} = \Dim \CC$ holds, then $\CC$ is an algebraic homotopy representation.
\end{proposition}

When the equality $\un{n} = \Dim \CC$ holds, we say that $\CC$ is a \emph{tight} complex.
 In general, $\un{n}(K) \leq \Dim \CC(K)$ for each $K \in \cF$, and one would expect obstructions to finding a tight complex which is chain homotopy equivalent to a given $R$-homology $\un{n}$-sphere. The following result shows the relevance of the internal homological conditions for this question. 
 
\begin{theorem}[{\cite[Theorem A]{Hambleton:2014}}]\label{thm:2014thma} Let  $G$ be a finite group, and  $\cF$ be a family of subgroups of $G$.  Suppose that 
 \begin{enumerate}
\item $R$ is a principal ideal domain,
\item $\un{n}\colon \Sub \to \bZ$ is a conjugation-invariant function supported on $\cF$, and
\item $\CC$ is a finite chain complex of free $\RG_G$-modules which is an
$R$-homology $\un{n}$-sphere.
\end{enumerate}
Then $\CC$ is chain homotopy equivalent to a finite
free chain complex $\DD$ satisfying $\un{n}=\Dim \DD$ if and only if $\CC$ is an algebraic homotopy representation. 
\end{theorem}
Theorem \ref{thm:2014thma} was motivated by \cite[Theorem 8.10]{Hambleton:2013}, which states that a finite chain complex of free $\bbZ\G_G$-modules can be realized by a geometric $G$-CW-complex if it is a tight homology $\un{n}$-sphere such that $\un{n} (H) \geq 3$ for all $H \in \cF$.  Upon combining these two statements, we get the following geometric realization result.

\begin{corollary}[{\cite[Corollary B]{Hambleton:2014}}]\label{cor:2014corb}
Let $\CC$ be a finite  chain complex of free $\ZG_G$-modules which is a
homology $\un{n}$-sphere. 
If $\CC$ is an  algebraic homotopy representation, 
and  in addition, if $\un{n}
(K) \geq 3$ for all $K \in \cF$, then there is a finite $G$-CW-complex $X$, with isotropy in $\cF$, such that $\uCZ{X}$ is chain homotopy equivalent to
$\CC$ as chain complexes of $\ZG_G$-modules. 
\end{corollary}

We conclude this section by stating an algebraic version of a well-known theorem in transformation groups:  the  dimension function of a homotopy representation satisfies certain conditions called the Borel-Smith conditions (see  
\cite[Thm. 2.3 in Chapter XIII]{Borel:1960} or \cite[III.5]{Dieck:1987}).

\begin{theorem}[{\cite[Theorem C]{Hambleton:2014}}] \label{thm:IntroBorel-Smith}
Let $G$ be a finite group, 
$R=\bbF_p$, and $\cF$ be a given family of subgroups of $G$. If $\CC$ is a finite projective chain complex over $\RG _G$, which is an $R$-homology
$\un{n}$-sphere, then the function $\un{n}$ satisfies the Borel-Smith conditions at the prime $p$.  
\end{theorem}

\begin{remark}
 The fact that the dimension function of  an algebraic $\un{n}$-homology sphere satisfies the Borel-Smith conditions suggests that more of the classical results on finite group actions on spheres might hold for finite projective chain complexes over a suitable orbit category. For example, one could ask for an algebraic version of the results of Dotzel-Hamrick  \cite{Dotzel:1981}
on $p$-groups. 
\end{remark}

\begin{example}\label{ex:Qdp}
An important test case for groups acting  on spheres, or on products of spheres \cite{Adem:2001}, is the rank two group $\Qd(p)=(\cy p \times \cy
p)\rtimes SL_2(p)$. At present, it is not known whether $\Qd(p)$ can act freely on a product of  two spheres, but  \" Unl\" u \cite{Unlu:2004} showed that 
$\Qd(p)$ does not act on a finite complex homotopy equivalent to a sphere with rank one isotropy.  The question of whether or not $\Qd(p)$ acts freely on a finite CW-complex homotopy equivalent to a product of two spheres is one of the most important open problems on free actions on products of spheres.
\end{example}

 As an application of Corollary \ref{cor:2014corb}, we can show that such a finite projective chain complex over $\RG_G$ does not exist for the group $G=\Qd(p)$ with respect to the family $\cF$ of rank one subgroups (see \cite[Proposition 5.14]{Hambleton:2014}).
This is an important group theoretic constraint  on the existence question for geometric homotopy representations with rank one isotropy.

One of the main ideas in the proof of Theorem \ref{thm:IntroBorel-Smith} is the reduction of a given chain complex of $R\G_G$-module $\CC$ to a chain complex over $R\G_{K/N}$ for a subquotient $K/N$ 
appearing in the Borel-Smith conditions. For this reduction, we introduce \emph{inflation} and \emph{deflation} of modules over the orbit category, via restriction and induction associated to a certain functor $F$ (see Section \ref{sec:seven}). Then we use spectral sequence arguments to conclude that the conditions given in the Borel-Smith conditions hold for these reduced chain complexes over $R\G _{K/N}$.

\section{ Inflation and deflation of chain complexes}\label{sec:seven}
In this section we define the two general operations on chain complexes mentioned in the last section, which are used in the proof of Theorem \ref{thm:IntroBorel-Smith}.  For a finite $G$-CW complex $X$ which is a mod-$p$ homology sphere, the Borel-Smith conditions can be
proved using a reduction argument to certain $p$-group subquotients (compare \cite[III.4]{Dieck:1987}). 
For a subquotient $K/L$, the reduction comes from considering the
fixed point space $X^L$ as a $K$-space. To do a similar reduction for chain complexes over $\RG _G$, we first introduce a new functor for $\RG _G$-modules, called the \emph{deflation} functor. We will introduce this functor as a restriction functor between corresponding module categories. For this discussion $R$ can be taken as any commutative ring with $1$, and $\cF_G$ is any family of subgroups subject to the extra conditions we assume during the construction.

Let $N$ be a normal subgroup of $G$. We define a functor $$F \colon
\G _{G/N} \to \G _G $$ by considering a $G/N$-set (or $G/N$-map) as
a $G$-set (or $G$-map) via composition with the quotient map $G \to G/N$. For this definition to make sense, the families $\cF
_{G/N}$ and $\cF _G$ should satisfy the property that if $K\geq N$ is such that $(K/N) \in
\cF _{G/N} $, then $K \in \cF _G$. Since we always
assume the families are nonempty, the above assumption also implies
that $N \in \cF_G$. For notational simplicity from now on, let us denote $K/N$ by $\overline K$ for every $K \geq N$.

If a family $\cF_G$ is already given, we will always take $\cF
_{G/N} = \{ \overline  K \vv  K \geq N \text{ and } K \in \cF_G \}$ and the condition
above will be automatically satisfied. We also assume that $N \in \cF _G$ to have a nonempty family for
$\cF _{G/N}$.

The functor $F$ gives rise to two functors  (see \cite[9.15]{Luck:1989}): 
$$\res _F \colon  \Mod\text{-}\RG_{G} \to \Mod\text{-}\RG_{G/N}$$ and
$$\ind _F \colon  \Mod\text{-}\RG_{G/N} \to \Mod\text{-}\RG_{G}\ .$$
The first functor $\res _F$ takes a $\RG _G$-module $M$ to the $\RG
_{G/N}$-module $$\Def ^G _{G/N} (M):= M \circ F \colon  \G _{G/N} \to
\RMod.$$ We call this functor the {\em deflation functor}. Note that
$$(\Def ^G _{G/N} M ) (\overline K)=M(K).$$
The induction functor $\ind _F$ associated to $F$
is called the {\em inflation functor}, denoted by $\Inf ^G _{G/N} (M)$. For every $H \in \cF_G$, we have $$\Inf ^G _{G/N} (M) (H)= \Bigl 
(\bigoplus\nolimits _{\overline{K} \in \cF _{G/N}} M(\overline{K})\otimes _{R W_{\overline G} (\overline{K})} R\Map_G (G/H, G/K) \Bigr ) / \sim$$
where the relations come from the tensor product over $R\G_{G/N}$ (see \cite[Definition 9.12]{Luck:1989}). In general, it can be difficult to calculate $\Inf ^G _{G/N} M$ for an arbitrary $R\G _{G/N}$-module $M$. In the case where $M$ is a free $R\G_{G/N}$-module we have the following lemma.

\begin{lemma}\label{lem:inflation} Let $X$ be a finite $G/N$-set. Then,
we have $$\Inf ^G _{G/N} \uR{X}= \uRb{(\Inf ^G _{G/N} X)}.$$ 
\end{lemma}

\begin{proof} It is enough to show this when $X=\oG/\oK$ for some $K\leq G$ such that $K \geq N$.  In this case, $\uRb{(\oG/\oK)}$ is isomorphic to $E_{\oK}P_{\oK}$ where $P_{\oK}= R[W_{\oG}(\oK)]$,  where $E_{\oK}(-)$ is defined as the \emph{induction} operation $\Ind _{j} (-)$ for the induction functor $j\colon  R[W_{\oG} (\oK)]\to R\G_{G/N}$ (see \cite[9.30]{Luck:1989}). We have 
$$ \Inf ^G _{G/N} \uRb{(\oG/\oK)}=\Inf ^G _{G/N} E_{\oK} P_{\oK}= \Ind _{F} \Ind _{j} P_{\oK}=\Ind _{F\circ j} P_{\oK} $$
where $F\colon  \G_{G/N} \to \G _G $ is the functor defined above. Since $W_{\oG}(\oK)\cong W_G (K)$, after suitable identification, the composition $F\circ j$ becomes the same as the inclusion functor $i\colon W_G (K) \to \G_G$, so we have $$\ind _{F \circ j}P_{\oK}=E_{K} RW_G (K) =\uRb{G/K}$$ as desired.
\end{proof}

By general properties of restriction and induction
functors associated to a functor $F$, the functor $\Def ^G _{G/N}$
is exact and $\Inf ^G _{G/N}$ respects projectives  (see \cite[9.24]{Luck:1989}). The deflation functor has the following formula for free modules.

\begin{lemma}\label{lem:deflation} Let $X$ be a $G$-set. Then,
we have $$\Def ^G _{G/N} \uR{X}= \uRb{(X^N)}$$
 In particular, if  $H
\in \cF_G$ implies $HN \in \cF_G$, then the functor $\Def ^G _{G/N}$
respects projectives.
\end{lemma}

\begin{proof} For every $K\in \cF _G$  such that $K\geq N$, we have
\begin{equation*}
\begin{split}
(\Def ^G _{G/N} \uR{X})(\overline K)&=\uR{X} (K)=R[X^K] 
=R[(X^N)^{K/N}]=\uRb{(X^N)} (\overline K). \\
\end{split}
\end{equation*}
Note that $(G/H)^N= G/HN$ as a 
$G/N$-set. If $H \in \cF_G$
implies $HN \in \cF_G$, then by assumption  $\overline{HN} \in \cF_{G/N}$.
Hence $\uRb{{((G/H)^N)}}$ is free as an $R\G_{G/N}$-module and $\Def ^G_{G/N}$ respects
projectives.
\end{proof}

\section{ The first non-linear example}\label{sec;eight}
In joint work \cite{Hambleton:2013} with Semra Pamuk, we established a general framework for studying $G$-CW-complexes via the orbit category. As an application of our methods, we considered the first non-trivial case of the existence problem: does there exist a rank two finite group $G$ which admits  finite $G$-CW-complex $X\simeq S^n$ with rank one isotropy, but no orthogonal linear
representations $V$ such that $S(V)$ has rank one isotropy.

If $G$ is a finite $p$-group of rank 2, then there exist orthogonal linear
representations $V$ so that $S(V)$ has rank 1 isotropy (see
\cite{Dotzel:1981}). If $G$ is not of prime power order,
representation spheres with rank 1 isotropy do not exist in general:
a necessary condition is that $G$ has a $p$-effective character for
each prime $p$ dividing $|G|$ (see \cite[Thm.~47]{Jackson:2007}). In
\cite[Prop.~48]{Jackson:2007} it is claimed that this condition is also
sufficient for an affirmative answer to the $G$-CW question above,
but the discussion on \cite[p.~831]{Jackson:2007} does not provide a
construction for $X$.

The first non-trivial case is the permutation group
$G=S_5$ of order 120, which has rank 2 but no linear action with rank 1
isotropy on any sphere, although it does admit $p$-effective
characters for $p=2,3,5$.

\begin{theorem}[{\cite[Theorem A]{Hambleton:2013}}]\label{thm:hpythma} The permutation group $G=S_5$ admits a finite $G$-CW-complex $X \simeq S^n$, such that $X^H \neq \emptyset$ implies that
$H$ is a rank 1 subgroup of $2$-power order.
\end{theorem}

\begin{remark}  
It is an interesting problem for future work to decide if the group
$G=S_5$ can act \emph{smoothly} on  $S^n$ with rank 1 isotropy.
\end{remark}

In order to prove this result we developed further techniques for chain complexes over the
orbit category, which may have some independent interest. A
well-known theorem of Rim \cite{Rim:1959} shows that a module $M$ over
the group ring $\bZ G$ is projective if and only if its restriction
$\res^G_PM$ to any $p$-Sylow subgroup is projective. For modules over the
orbit category we have a similar statement localized at $p$.
 
\begin{theorem}[{\cite[Theorem 3.9]{Hambleton:2013}}] 
Let $G$ be a finite group and let $R=\bZp$. Then an
$\RG_G$-module $M$ has a finite projective resolution with respect to
a family of $p$-subgroups if and only if its restriction $\res^G_PM$
has a finite projective resolution over any $p$-Sylow subgroup $P\leq G$.
\end{theorem}

\begin{remark} For modules over the group ring $RG$, those having
finite projective resolutions are already projective. Over the orbit
category, these two properties are distinct.
\end{remark}

The construction of the $G$-CW-complex $X$ for $G=S_5$ and the proof of Theorem \ref{thm:hpythma} is carried
out in several steps (see \cite[Section 9]{Hambleton:2013}).  We first construct finite projective chain complexes
$\CC^{(p)}$ over the orbit categories $\RG_G$, with $R = \bZp$,  separately for the prime $p = 2, 3, 5$ dividing $|G|$. In each case, the isotropy family  $\cF$ consists of
the rank 1 subgroups of $2$-power order in $G$. 

The chain complexes $\CC ^{(p)}$ all have the same \emph{dimension function}.  We prescribe a non-negative function
$\un{n}\colon \cF \to \bZ$, with the property that $\un{n}(K) \leq \un{n}(H)$ whenever $H$ is conjugate to a subgroup of $K$. Then, by construction, 
 each complex $\CC ^{(p)}$ has the $R$-homology of an $\un{n}$-sphere: 
for each $K\in \cF$, the
complexes $\CC^{(p)}(K)$ have homology  $H_i =R$ only in two
dimensions $i=0$ and $i= \un{n}(K)$. In other words, the complexes $\CC ^{(p)}$ are algebraic versions of tom Dieck's \emph{homotopy representations}.

In the case $p=2$, we start with the group $H = S_4$ acting by orthogonal rotations on
 the $2$-sphere. A regular $H$-equivariant triangulation of an inscribed cube or octahedron gives a finite projective chain complex over $\RG _H$. 
 Then we use a chain complex version of Theorem \ref{thm:hpythmc}, to lift
it to a finite projective complex over $\RG _G$. 
For $p=3$ and
$p=5$, the $p$-rank of $S_5$ is $1$, and there exists a periodic
complex over the group ring $RG$ (see Swan \cite[Theorem B]{Swan:1960}). 
We start with a periodic complex
over $RG$ and add chain complexes to this complex, for every
nontrivial subgroup $K \in \cF$, to obtain the required complex $\CC^{(p)}$ over $\RG_G$. 

 We use the theory of algebraic Postnikov
sections by Dold \cite{Dold:1960} to glue the complexes together to form
a finite projective $\ZG_G$-chain complex. We complete the chain
complex construction by varying the finiteness obstruction  to obtain
a complex of free $\ZG_G$-modules, and then we prove a realization
theorem Corollary \ref{cor:2014corb} to construct the
required $G$-CW-complex $X \simeq S^n$.

\section{ Mackey structures over the orbit category}
\label{sect:nine}

Another useful feature of homological algebra over group rings is
the detection of group cohomology by restriction to the $p$-Sylow
subgroups. Here is an important concept in group cohomology (see for
example \cite{Symonds:2004}).

\begin{definition} For a given prime $p$, we say that a subgroup
$H \subseteq G$ \emph{controls $p$-fusion} provided that
\begin{enumerate}
\item $p \nmid |G/H|$, and
\item whenever $Q \subseteq H$ is a $p$-subgroup, and there exists $g
\in G$
such that $Q^g:=g^{-1}Q g \subseteq H$, then $g = ch$ where $c \in
C_G(Q)$ and $h \in H$.
\end{enumerate}
\end{definition}

One reason for the importance of this definition is the fact that
the restriction map $$H^*(G; \Fp) \to H^*(H;\Fp)$$ is an isomorphism
if and only if $H$ controls $p$-fusion in $G$ (see \cite{Mislin:1990a},
\cite{Symonds:2004}). 
We have the following generalization of this result based on the theory of
 Mackey functors of cohomological type over the
orbit category (with respect to any family $\cF$).

\begin{theorem}[{\cite[Theorem C]{Hambleton:2013}}] \label{thm:hpythmc}
Let $G$ be a finite group, $R=\bZp$, and $H\leq G$ a
subgroup
which controls $p$-fusion in $G$. If $M$ is an $\RG_G$-module and
$N$ is a cohomological Mackey functor, then the restriction map
$$\res^G_H\colon \Ext^n_{\RG_G}(M,N) \to  \Ext^n_{\RG_H}(\res^G_HM,
\res^G_HN)$$
is an isomorphism for $n\geqslant 0$, provided that the centralizer
$C_G(Q)$ of any $p$-subgroup $Q \leq H$, with $Q \in \cF$, acts trivially on $M(Q)$ and
$N(Q)$.
\end{theorem}

The proof of Theorem \ref{thm:hpythma} uses some structural and computational facts about the $\Ext$-groups over the orbit category. These $\Ext$-groups are the domain of the $k$-invariants in the theory of algebraic Postnikov
sections \cite{Dold:1960}. The exposition is based on \cite[\S 4]{Hambleton:2013}, following Cartan-Eilenberg \cite{Cartan:1999} and  tom Dieck \cite[\S II.9]{Dieck:1987}.
%(see also \cite{jackowski-mcclure-oliver2}, \cite{grodal1}). 

We have seen that the category of right $\RG$-modules has enough
projectives to define the bifunctor
$$\Ext^*_{\RG}(M,N)=H^*(\Hom_{\RG}(\PP,N))$$
 via any projective
resolution $\PP\twoheadrightarrow M$ 
 (see
  \cite[Chap.~III, \S 17]{Luck:1989}, \cite[Chap.~III.6]{Mac-Lane:1995}). 

%The following property
%is also useful (see  L\"uck \cite[17.21]{lueck3}).
%
%\begin{lemma}\label{lem:E_x property}
%If $\G$ is a 
%free EI-category, then
%$\Ext^*_{\RG}(E_xM,N)\cong\Ext^*_{R[x]}(M,\Res_xN)$.
%\end{lemma}
%
%\begin{proof}  
%Take a projective resolution $\PP$ of $M$.  Since $\G$ is free, the extension functor $E_x$ is exact \cite[16.9]{lueck3}. In addition,  $E_x$ preserves projectives and is adjoint to the restriction functor $\Res_x$ by Proposition \ref{prop: induction-restriction}.
%Therefore
%$$\cdots\rightarrow E_x P_n \rightarrow
%\cdots \rightarrow E_xP_1\rightarrow E_xP_0\rightarrow
%E_xM\rightarrow 0$$ 
%is a projective resolution of $E_xM$, and applying
%$\Hom$ over the orbit category gives
%\begin{eqnarray*}
%\Ext^n_{\RG}(E_xM,N)&=& H^n(\Hom_{\RG}(E_x\PP, N))\\
%&\cong& H^n(\Hom_{R[x]}(\PP, \Res _x N))=\Ext^n_{R[x]}(M,\Res_xN). \qedhere\end{eqnarray*}
%\end{proof}

In the rest of  this section, we assume that $\G_G = \OrG$ for a
finite group $G$, where $\cF$ is a family of subgroups in $G$. 
Note that $\G_G$ is both finite and free as an EI-category.
 If there are
two groups $H \leq G$, we use the notations $\G _G = \OrG$  for the orbit category with respect to the family $\cF$,  and $\G
_H =\OrH$  for the orbit category with respect to the family  $\cF_H = \{ H \cap K \vv K \in \cF\}$.

We now recall the definition of a Mackey functor (following
Dress \cite{Dress:1975}). Let $G$ be a finite group and $\cD(G)$ denote
the Dress category of finite $G$-sets and $G$-maps.  A bivariant functor
$$M=(M^*,M_*)\colon \cD(G) \rightarrow \RMod$$
consists of a contravariant functor
$$ M^*\colon\cD(G)\rightarrow \RMod$$ and a covariant
functor
$$M_*\colon \cD(G)\rightarrow \RMod.$$ The functors
are assumed to coincide on objects.  Therefore, we write
$M(S)=M_*(S)=M^*(S)$ for a finite $G$-set $S$.  If $f\colon
S\rightarrow  T$ is a morphism, we often use the notation
$f_*=M_*(f)$ and $f^*=M^*(f)$. If $S = G/H$ and $T=G/K$ with $H\leq
K $ and $f\colon G/H \to G/K$ is given by $f(eH)=eK$, then we use the
notation $f_* = \ind_H^K$ and $f^* = \res_H^K$.

\begin{definition}[{Dress \cite{Dress:1975}}]
A bivariant functor is called a \emph{Mackey functor} if it has the following
properties:
\begin{enumerate}\addtolength{\itemsep}{0.2\baselineskip}
\item [(M1)]For each pullback diagram $$\xymatrix{X\ar[r]^h\ar[d]_g
&Y\ar[d]^k\\S\ar[r]_f&T }$$ of finite $G$-sets, we have $h_*\circ
g^*=k^*\circ f_*$.
\item [(M2)] The two embeddings $S\rightarrow  S\disjointunion
T\longleftarrow  T$ into the disjoint union define an isomorphism
$M^*(S\disjointunion T)\cong M^*(S) \oplus M^*(T)$.
\end{enumerate}
\end{definition}

\begin{remark} There is a functor $\cO(G) \to \cD(G)$ defined on
objects by $H \mapsto G/H$ for every subgroup $H \leq G$, and as the
identity on morphism sets. By composition, any contravariant functor $\cD(G) \to \RMod$ gives a
right $\RG_G$-module,  with respect to any given family of subgroups $\cF$ of $G$.

In the  statement of Theorem \ref{thm:Mackey structure}  we will use the examples
$\uR{S}\colon \cD(G) \to \RMod$, defined by $ \uR{S}(G/K) = R\Map _G(G/K ,S)$
 for every $K \in \cF$, where
$\Map _G(G/K ,S)$ denotes the set of $G$-maps from $G/K$ to $S$.
\end{remark}

The proof of Theorem \ref{thm:hpythmc} follows by showing that $H\mapsto
\Ext^*_{\RG_H}(M,N)$ has a cohomological Mackey functor structure
which is conjugation-invariant.  
First
we describe the Mackey functor structure on $\Hom_{\RG_{?}}(M,N)$.

\begin{theorem}[{\cite[Theorem 4.11]{Hambleton:2013}}]
\label{thm:Mackey structure} For a right $\RG_G$-module $M$ and a
Mackey functor $N$, let $$\uHom(M,N)\colon \cD(G) \to \RMod$$ denote
the function defined by $S\mapsto \Hom_{\RG_G}(M\otimes_R\uR{S},N)$
for any finite $G$-set $S$. Then $\Hom_{\RG_?}(M,N)$ inherits a
Mackey functor structure.
\end{theorem}

\begin{proof}
We will first define the induction and restriction maps to see that
$\uHom(M,N)$ is a bifunctor. For $f\colon S\rightarrow  T$ a
$G$-map,  the \emph{restriction} map
$$f^*\colon  \Hom_{\RG_G}(M\otimes_R\uR{T},N)\rightarrow
\Hom_{\RG_G}(M\otimes_R\uR{S},N)$$ is the composition with
$M\otimes_R \uR{S}\xrightarrow{\id\otimes\tilde{f}}M\otimes_R
\uR{T}$ where $\tilde{f}$ denotes  is the linear extension of the
map induced by $f$. Since the functors $\uR{S}$ satisfy axiom (M2),
so does $\uHom(M,N)$.

For $f\colon S\rightarrow  T$ a $G$-map, we  define  the
\emph{induction} map
$$f_*\colon  \Hom_{\RG_G}(M\otimes_R\uR{S},N)\rightarrow
\Hom_{\RG_G}(M\otimes_R\uR{T},N)$$ in the following way: let
$\varphi _S\colon M\otimes_R \uR{S}\rightarrow  N$ be given. We
will describe the homomorphism $\varphi _T=f_*(\varphi _S)$.
$$\varphi _T(V)(x\otimes\alpha)=F_*\Bigl (\varphi _S(U)(F^*(x)\otimes
\beta ) \Bigr )$$ for $x\in M(V)$ and $\alpha\colon V\rightarrow  T$,
where $U$, $\beta $ and $F$ are given by the pull-back
$$\xymatrix{U\ar[r]^\beta \ar[d]_F&S\ar[d]^f\\V\ar[r]_\alpha&T}$$
It is easy to check that this formula for $\varphi _T$ gives an
$\RG_G$-homomorphism, using the assumption that $N$ is a Mackey functor.
For the remaining details, see \cite[Theorem 4.11]{Hambleton:2013}.
\end{proof}

The Mackey structure on $\uHom(M,N)$ extends without difficulty to chain complexes over the orbit category.
\begin{proposition}[{\cite[Proposition 4.13]{Hambleton:2013}}]
Let $\CC$ be a chain complex of right $\RG_G$-modules and $N$ be a
Mackey functor. Then, the cochain complex
$$\bC^*=\uHom(\CC,N)$$ with the differential $\delta\colon
\uHom(C_i,N)\rightarrow \uHom(C_{i+1},N)$ given by
$\delta(\varphi )=\varphi \circ\bd$ is a cochain complex of Mackey
functors.
\end{proposition}

\begin{corollary}
Let $M$ be an $\RG_G$-module and $N$ be a Mackey functor.  Then,
$$\uExt^*(M,N)$$ has a Mackey functor structure.  As a Mackey
functor $\uExt^*(M,N)$ is equal to the homology of the
cochain complex of Mackey functors $\uHom(\PP,N)$ where $\PP$
is a projective resolution of $M$ as an $\RG_G$-module.
\end{corollary}
\begin{proof} To compute the $\Ext$-groups, 
note that $S\mapsto \PP\otimes_R \uR{S}$ is a projective resolution
of the module $S\mapsto M \otimes_R \uR{S}$, for every finite
$G$-set $S$. 
\end{proof}
\begin{remark}
It follows that a version of the Eckmann-Shapiro isomorphism
$$\Ext^*_{\RG_G}(M\otimes \uR{G/H},N)\cong \Ext^*_{\RG_H}(\Res_H^G
M,\Res_H^G N)$$ holds for the $\Ext$-groups over the orbit category (compare  \cite[2.8.4]{Benson:1998}).
\end{remark}

\begin{remark} If $N$ is a Green module over a Green ring $\cG$, then
the Mackey functor $\uExt^*(M,N)$ also inherits a Green module
structure over $\cG$. The basic formula is a pairing $$\cG(S) \times
\Hom_{\RG_?}(M\otimes_R\uR{S},N) \to
\Hom_{\RG_?}(M\otimes_R\uR{S},N)$$ induced by the Green module
pairing $\cG \times N \to N$. For any $z\in \cG(S)$, $x\in M(V)$,
and $\alpha \colon V \to S$, we define
$$(z\cdot \varphi_S)(V)(x\otimes \alpha) = \alpha^*(z)\cdot
\varphi_S(V)(x\otimes \alpha)$$
for any $\varphi_S(V)\colon M(V) \otimes_R R\Mor(S, V) \to N(V)$.
The check that this pairing gives a Green module structure is left
to the reader. \qed
\end{remark}

\section{Existence Results for rank two simple groups}\label{sec:ten}
In \cite{Hambleton:2016} we studied the existence problem for rank two finite groups $G$, with respect to the family $\cH$ of all rank one $p$-subgroups of $G$. We showed  that a rank two finite group $G$ which satisfies certain explicit group-theoretic conditions  admits a finite $G$-CW-complex $X\simeq S^n$ with isotropy in $\cH$, whose fixed sets are homotopy spheres. Our construction provided an infinite family of new non-linear $G$-CW-complex examples. 

Apart from the input from P.~A.~Smith theory,  there is another group theoretical \emph{necessary} condition related to fusion properties of the Sylow subgroups. This condition involves the rank two finite group $\Qd(p)$, since any rank two finite group which includes $\Qd(p)$ as a subgroup cannot admit such actions. 

More generally, we say $\Qd(p)$ is $p'$-\emph{involved in $G$} if there exists a subgroup $K \leq G$, of order prime to $p$, such that $N_G(K)/K$ contains a subgroup isomorphic to $\Qd(p)$. It follows that, whenever there exists a finite $G$-CW-complex $X\simeq S^n$ with rank one isotropy, the group $\Qd(p)$ is not $p'$-involved in $G$,  for any odd prime $p$. If $G$ does not $p'$-involve $\Qd(p)$, for $p$ odd, we say that $G$ is called \emph{$\Qd(p)$-free}.
 
\begin{definition}\label{def:pEffective}
Let $G_p$ be a Sylow $p$-subgroup of $G$. A character  $\chi : G_p\to \bbC$  \emph{respects fusion} in $G$ if  
$\chi (gxg^{-1})=\chi(x)$ whenever $gxg^{-1} \in G_p$ for some $g \in G$ and $x \in G_p$. 
A finite group $G$ is said to have a $p$-\emph{effective character} 
if it has a character $\chi\colon G_p \to \bbC$ which respects fusion and  satisfies $\la \chi  |_E, 1_E\ra = 0$ 
for every elementary abelian $p$-subgroup $E\leq G$ with rank equal to the $p$-rank of the group. 
\end{definition}

We recall that in \cite[Theorem 47]{Jackson:2007} it was proved that a rank two group $G$ has a $p$-effective character if and only if $p=2$, or $p$ is odd and $G$ is  $\Qd(p)$-free. We will use these characters to obtain finite $G$-homotopy representations with rank one prime power isotropy, assuming an additional closure property at the prime $p=2$.

Let $\PG$  denote the set of primes $p$ such that $\rk_p(G)=2$. Let $\cH_p$ denote the family of all rank one  $p$-subgroups $H \leq G$, for $p \in \PG$, and  let $\cH = \bigcup\, \{ H \in \cH_p \vv p\in \PG\}$.

\begin{theorem}[{\cite[Theorem A]{Hambleton:2016}}] \label{thm:main2016} Let $G$ be
rank two finite group satisfying  the following two conditions: 
\begin{enumerate}
\item $G$ is    $2$-regular if $2 \in \PG$, and $G$ is $\Qd(p)$-free for all $p\in \PG$ with $p>2$;  
% either $\Omega_1(Z(G_2))$ is strongly closed in $G_2$ with respect to $G$, or $G$ satisfies the rank one intersection property. 
\item  If $1 \neq H \in \cH_p$, then $\rk_q(N_G(H)/H) \leq 1$ for every prime $q\neq p$.  
\end{enumerate}
Then there exists a finite    $G$-homotopy representation $X$ with isotropy in $\cH$. 
\end{theorem}

This result is an extension of our earlier joint work with Semra Pamuk  \cite{Hambleton:2013} for the  permutation group $G = S_5$ of order 120,  and gives a new proof in that case:  for $G=S_5$, the set $\PG$ includes only the  prime $2$ and it can be easily seen that $G$ satisfies the rank one intersection property. 
 The second condition above also holds since all $p$-Sylow subgroups of $S_5$ for odd primes are cyclic.

  More generally, using Theorem \ref{thm:main2016}, we obtain many new non-linear $G$-CW-complex examples.  
In particular, we show that the alternating group $A_6$ admits finite $G$-CW-complexes $X \simeq S^n$ with rank one isotropy  (see  \cite[Example 6.5]{Hambleton:2016}). We also discuss  why $G=A_7$ can not admit such actions if we require $X$ to be an $G$-homotopy  representation with rank one prime power isotropy (see  \cite[Example 6.7]{Hambleton:2016}). In fact we show exactly which  of the rank two simple groups (see the list in \cite[p.423]{Adem:2001}) can admit such actions. 

\begin{theorem}[{\cite[Theorem C]{Hambleton:2016}}]\label{hy2016thmC}
Let $G$ be a finite simple group of rank two. Then there exists a finite   $G$-homotopy representation with rank one isotropy of prime power order  if and only if 
$G$ is one of the following: (i) $PSL_2(q)$, $q\geq 5$, (ii)  $PSL_2 (q^2)$, $q \geq 3$, 
(iii) $ PSU_3(3)$, or (iv) $ PSU_3(4)$.
\end{theorem}

We remark that  $G = PSL_3(q)$, $q$ odd, and $G = PSU_3(q)$, with $9 \mid (q+1)$, are the rank two simple groups that are not $\Qd(p)$-free at some odd prime. The remaining simple groups $G=PSU_3(q)$, $q\geq 5$, are eliminated by  the Borel-Smith conditions (see \cite[Section 7]{Hambleton:2016}). 
 The groups $PSU_3(3)$ and $PSU_3(4)$ have a linear actions on  spheres with rank one prime power isotropy. We note that the group $G=PSU_3(3)$  does not satisfy the rank one intersection property
 (see \cite[Example 6.8]{Hambleton:2016}). 
 
 We will obtain Theorem \ref{thm:main2016} from a more general technical result, 
which accepts as input a compatible collection of representations defined on all $p$-subgroups of $G$, for a given set of primes, 
and produces a finite $G$-CW-complex. 
In principle,  this method could be used to construct other interesting non-linear examples for finite groups with specified $p$-group isotropy.

 The proofs build on the theory of algebraic homotopy representations and a local-global construction (see \cite[Sections 3-6]{Hambleton:2016}). We construct $p$-local chain complexes where  the isotropy subgroups are $p$-groups, and then add homology to these local models so that these modified local complexes $\CC \up$ all have exactly the same dimension function.  Results established in \cite{Hambleton:2013} are used to  glue these algebraic complexes together over $\bZ\G_G$, and  then  to eliminate a finiteness obstruction.We then combine these ingredients to
 complete the proof.  Finally, we discuss the necessity of the conditions in Theorem \ref{thm:main2016} and provide  a nonlinear action for the group $G=A_6$.
 
 \begin{remark} Examples of $G$-CW-complexes homotopy equivalent to spheres provide input into the Adem-Smith construction \cite{Adem:2001} of free actions of finite groups on products of spheres. Here is a special case of one of their main results.
 \begin{theorem}[Adem and Smith {\cite[Theorem 1.4]{Adem:2001}}] Let $G$ be a finite group, and let $X$ be a finite dimensional $G$-CW-complex with rank one isotropy, which is homotopy equivalent to an  $n$-sphere, $n>1$. Then $G$ acts freely on a finite CW-complex homotopy equivalent to $S^n \times S^m$, for some $m >0$.
 \end{theorem}

 Our results in Theorem \ref{hy2016thmC} extend the list of examples in \cite[Theorem 1.7]{Adem:2001} for the rank two finite simple groups. 
 \end{remark}
 
 \section{Group actions on homotopy spheres with prime power isotropy}\label{sec:eleven}
 
 In the less restrictive setting (C), where we aim to construct $G$-CW-complexes homotopy equivalent to a sphere 
 (rather than homotopy representations), there is a complete result for rank two finite 
 group actions. 
 
 \begin{theorem}[{\cite[Theorem A]{Hambleton:2017}}]\label{thm:HomotopySpheresA}
 Let $G$ be a finite group of rank two. If $G$ admits a finite $G$-CW-complex $X\simeq S^n$ with rank one  isotropy then $G$ is $\Qd(p)$-free. Conversely, if $G$ is $\Qd(p)$-free, then there exists a finite $G$-CW-complex $X\simeq S^n$ with rank one prime power isotropy.
 \end{theorem}

The necessity of the $\Qd(p)$-free condition was established in \cite[Theorem 3.3]{Unlu:2004} and \cite[Proposition 5.4]{Hambleton:2016}.
For the other direction, we introduce a more general method for constructing finite group
actions on homotopy spheres with prime power isotropy.  The input of this construction method is a $G$-invariant family representations of Sylow subgroups.

\begin{definition}\label{def:Ginvariant}  For each prime $p$ dividing the order of $G$, let $G_p$ denote a fixed Sylow $p$-subgroup of $G$.
A representation $V_p$ of $G_p$ is called a \emph{Sylow representation}.  A family $\{ V_p\}$  of Sylow 
representations over all primes $p$ is \emph{$G$-invariant} if the character of $V_p$ respects fusion in $G$ and for all $p$, the dimension of 
$V_p$ is equal to a fixed positive integer $n$. 
\end{definition}

Let $\cP$ denote the family of all subgroups of $G$ which has prime power order. Given a $G$-invariant family of Sylow representations $\{ V_p \}$, we construct a $G$-equivariant spherical fibration $q\colon E \to B$ over a contractible $G$-space $B$ with isotropy in $\cP$ such that  for every $x\in \Fix(B, G_p) = B^{G_p}$, the fiber $q^{-1} (x)$ is $G_p$-homotopy equivalent to $S (V_p ^{\oplus k})$ for some $k\geq 1$ (see \cite[Theorem 3.4]{Hambleton:2017}). The total space of this $G$-fibration admits a $G$-map $$f_0 \colon \coprod _p G \times _{G_p} S(V_p ^{\oplus k}) \to E.$$
By adapting the $G$-CW-surgery techniques introduced by Oliver and Petrie \cite{Oliver:1982} to this $G$-map, we obtain a finite  $G$-CW-complex $X \simeq S^{2kn-1}$ whose restriction to Sylow $p$-subgroups resembles the linear spheres $S(V_p^{\oplus k})$. In particular, we prove the following theorem (see \cite[Definition 3.6]{Hambleton:2017} for the definition of $p$-local $G$-equivalence)

\begin{theorem}[{\cite[Theorem B]{Hambleton:2017}}]\label{thm:MainConstruction} Let $G$ be a finite group.   Suppose that $\{V_p\colon G_p \to U(n)\}$ is  a $G$-invariant family of Sylow representations. Then there exists a positive integer $k\geq 1$ and a finite $G$-CW-complex $X\simeq S^{2kn-1}$ with prime power isotropy,  such that  the $G_p$-CW-complex $\Res ^G _{G_p} X$ is $p$-locally $G_p$-equivalent to $S(V_p ^{\oplus k} )$, for every prime $p \mid |G|$.
\end{theorem} 

This theorem was stated by Petrie \cite[Theorem C]{Petrie:1979}  in a slightly different form and a sketched proof was provided. Related results were proved by tom Dieck   (see \cite[Satz 2.5]{Dieck:1982},  \cite[Theorem 1.7]{Dieck:1985a}).  The $G$-CW-complexes obtained using Theorem \ref{thm:MainConstruction} are not homotopy representations in general. The following example shows that for some $G$-invariant family of Sylow representations $\{ V_p\}$, it is not possible to obtain homotopy representations which are $p$-locally $G_p$-equivalent to $S(V_p ^{\oplus k} )$ for any $k\geq 1$.

\begin{example} Let $G$ denote the dihedral group of order $2q$, with $q$ an odd prime. Let $V_2$ be a trivial representation of $G_2 = \cy 2$, and let $V_q$ be a free unitary representation of $G_q = \cy q$, such that $\dim V_2 = \dim V_q$. Then Theorem \ref{thm:MainConstruction}  shows that there exists a finite $G$-CW-complex $X \simeq S^{m}$, with $\Fix(X, G_2) \simeq S^{m}$ ($2$-locally),  and  $\Fix (X, G)=\Fix(X, G_q) = \emptyset$, for some  integer $m=2kn-1$. However, these conditions imply that $\dim X > m$ by \cite[Proposition 2.10]{Hambleton:2014} since in this case $\un{n} (1) =\un{n} (G_2)=m$ but $\un{n}(G)\neq m$ (compare \cite[Theorem 4.2]{Straume:1981}). 
\end{example}

To prove Theorem \ref{thm:HomotopySpheresA} as a consequence of the construction given in Theorem \ref{thm:MainConstruction},
we use a theorem by Jackson \cite{Jackson:2007} which states that if $G$ is a rank two finite group which is $\Qd (p)$-free, 
then it has a Sylow representation $V_p: G_p \to U(n)$ with a $p$-effective character. By taking multiples of these representations if necessary, 
we can assume that they have a common dimension. This gives a $G$-invariant family $\{ V_p\}$ 
such that $\la V_p |_E , 1_E \ra =0$, for every elementary abelian $p$-subgroup $E\leq G$ with $\rk E=2$. 

Applying Theorem \ref{thm:MainConstruction}, we obtain a finite $G$-CW-complex  $X$ homotopy equivalent to a sphere $S^{2kn-1}$, for some $k\geq 1$, such that $\Res^G _{G_p} X$ is $p$-locally $G_p$-equivalent to $S(V_p ^{\oplus k} )$, 
for every prime $p$ dividing the order of $G$. In particular, for every $p$-subgroup $H \leq G$, the fixed point space $X^H$ has the same 
$p$-local homological dimension as the fixed point sphere $S(V_p  ^{\oplus k})^H$.  
Since $S(V_p) ^E= \emptyset$, we have $S(V_p  ^{\oplus k})^H = \emptyset$ 
for every subgroup $H \leq G$ with $\rk (H)=2$. Hence the isotropy subgroups of $X$ 
are all rank one subgroups of $G$ with prime power order.

\begin{remark} A similar idea for constructing spherical fibrations to obtain $G$-actions on homotopy spheres 
was introduced by Connolly and Prassidis \cite{Connolly:1989}. They proved the following result (conjectured by C.~T.~C.~Wall):
\begin{theorem}[Connolly and Prassidis {\cite[Corollary 1.4]{Connolly:1989}}]
A discrete group $\Gamma$ with finite virtual cohomological dimension acts freely and properly discontinuously on $\bbR ^m \times S^{n-1}$ for some $m$ and $n$,  if and only if $\Gamma$ is countable and $\widehat H^* (\Gamma; \bbZ)$  has periodic Farrell-Tate  cohomology.
\end{theorem}
Later, the finiteness condition was removed by Adem and Smith \cite[Corollary 1.3]{Adem:2001}, and hence the result holds for any countable discrete group with periodic cohomology.

In the same paper, Connolly and Prassidis \cite[Theorem C]{Connolly:1989} showed that if $\Gamma$ was also assumed to be a virtual Poincar\'e duality group (with a mild additional condition on the lattice of its finite subgroups), then there exists a finite 
Poincar\'e complex $X$ with $\pi_1(X) \cong \Gamma$, and universal covering space $\widetilde X \simeq S^{n-1}$, for some $n$. The first manifold examples of smooth co-compact free actions on $\bbR ^m \times S^{n-1}$, for some $m$ and $n$, were constructed in  \cite[Theorem 8.3]{Hambleton:1991}.
\end{remark} 
 
Another consequence of the construction given in Theorem \ref{thm:MainConstruction} is a realization theorem for dimension functions satisfying the Borel-Smith conditions. Let $G$ be finite group and $\cP$ denote the collection of all subgroups of $G$ with prime power order.
If $X$ is a finite (or finite-dimensional) $G$-CW-complex such that $X$ is homotopy equivalent to a sphere, then by 
Smith theory, for each $p$-group $H \leq G$, the fixed point subspace $X^H$ has mod-$p$ homology of a sphere 
$S^{\underline{n} (H)}$. For convenience, we define $\Dim_{\cP} X \colon \cP \to \bbZ$ to be the conjugation invariant function with values 
$(\Dim _{\cP} X) (H)=\underline{n} (H)+1$ for every $p$-subgroup $H \leq G$, over all primes dividing the order of $G$.
 This change ensures that $\Dim _{\cP}$ is an additive function.
In \cite{Reeh:2018}, the following theorem is proved:

\begin{theorem}\label{thm:mainApp} Let $G$ be a finite group, and let $f\colon \cP \to \bbZ$ be a monotone conjugation invariant 
function that satisfies the Borel-Smith conditions. Then there is an integer $N\geq 1$ and a finite $G$-CW-complex $X \simeq S^n$, with prime power isotropy, such that $\Dim_{\cP} X=Nf$.
\end{theorem}

 Theorem \ref{thm:mainApp} reduces the question of finding actions on homotopy spheres with restrictions on the rank of isotropy subgroups, to finding Borel-Smith functions that satisfy certain conditions. For example, Theorem \ref{thm:HomotopySpheresA} 
 directly follows from Theorem \ref{thm:mainApp} once we show that for every $\Qd(p)$-free rank two finite group $G$, there is a monotone Borel-Smith function $f\colon \cP \to \bbZ$ such that $f(H)=0$ for every $H \leq G$ with $\rk H=2$.

Theorem \ref{thm:mainApp} is also related to a question of Grodal and Smith  \cite{Grodal:2007}  on algebraic models for homotopy $G$-spheres. 
Let $G$ be a finite group, $\cH _p$ denote the family of
all $p$-subgroups in $G$, and $\Gamma_G=\Or_p(G)$ denote the orbit category over all $p$-subgroups of $G$. 
 A chain complex $\CC$ of $R\Gamma _G$-modules is called \emph{perfect} if it is finite-dimensional with $\CC_i$ a finitely generated projective $R\Gamma_G$-module for each $i$.  The dimension function of an $R$-homology $\underline{n}$-sphere $\CC$ over $R\Gamma_G$ is defined as a function $\Dim _{\cH_p} \CC\colon \cH_p \to \bbZ$ such that $(\Dim _{\cH_p} \CC)(H)=\underline{n} (H)+1$ for all $H \in \cH_p$.  It was mentioned in \cite{Grodal:2007}, and shown in \cite{Hambleton:2014} that if $\CC$ is a perfect complex which is an $\bbF_p$-homology $\underline{n}$-sphere, then the dimension function $\Dim _{\cH_p} \CC$ satisfies the Borel-Smith conditions (see \cite[Theorem C]{Hambleton:2014}). Grodal and Smith \cite{Grodal:2007} suggests that the converse also holds. 
 
\begin{question}[Grodal-Smith]\label{ques:GrodalSmith} Let $G$ be a finite group, and let $f$ be a monotone Borel-Smith function defined on $p$-subgroups of $G$. Is there then a perfect $\bbF_p \Gamma _G$-complex $\CC$, with dimension function $\Dim _{\cH_p} \CC=f$, which is an $\bbF _p$-homology $\underline{n}$-sphere ?
\end{question}

\section{Equivariant Moore spaces}\label{sec:twelve}

Let  $M$ be $\bbZ G$-module and $n\geq 0$ an integer. A $G$-CW-complex $X$ is called a \emph{Moore $G$-space of type $(M, n)$} if $\widetilde H_i(X ; \bbZ)=0$ for $i\neq n$ and $\widetilde H_n (X; \bbZ) \cong M$ as $\bbZ G$-modules. One of the classical problems in algebraic topology, due to Steenrod, asks whether every $\bbZ G$-module is realizable as the homology module of a Moore $G$-space.  Carlsson  \cite{Carlsson:1981a} showed that if $G$ is a finite group which contains $\bbZ/p\times \bbZ/p$ for some $p$, then there are non-realizable $\bbZ G$-modules.   The question of finding a good algebraic characterization of realizable $\bbZ G$-modules is still an open problem (see  \cite{Smith:1985,Smith:1985a} and  \cite{Benson:1987a}).

More generally, one can consider Moore $G$-spaces whose fixed point subspaces are also Moore spaces (as in  \cite{Yalcin:2017}).
\begin{definition} Let $\un{n} \colon \Sub \to \bbZ$ be a conjugation-invariant function. A $G$-CW-complex $X$ is called an \emph{$\un{n}$-Moore $G$-space over $R$} if for every $H \leq G$, the reduced homology group $\widetilde H_i(X ^H ; R)$ is zero for all $i\neq \un{n}(H)$. 
\end{definition}

When $\un{n}$ is the constant function with value $n$ for all $H\leq G$, the homology at dimension $n$ can be considered as a module over the orbit category $\Or G$. If $\widetilde{\un{H}}_n (X^?; R) \cong \un{M}$ as a module over $\Or G$, then $X$ is called a \emph{Moore $G$-space of type $(\un{M}, n)$}. When $R=\bbQ$ and $X^H$ is simply-connected for all $H \leq G$, the space $X$ is called a rational Moore $G$-space.  These spaces are studied extensively in equivariant homotopy theory (see  \cite{Doman:1988,Kahn:1986}).   

The concept of Moore $G$-spaces can also be generalized to allow the dimension function $\un{n}$ to be non-constant.  In \cite{Yalcin:2017}, the coefficient ring $R$ is taken as a field $k$ of characteristic $p$, and the group of $G$-Moore spaces $\cM_k (G)$ 
is defined as the Grothendieck group of the monoid of finite $\un{n}$-Moore $G$-spaces over $k$ with addition given by the join operation.
When $G$ is a $p$-group, $\cM _k (G)$ can be related to the Dade group $D_k(G)$, the group of endo-permutation $kG$-modules, via an exact sequence (see \cite[Theorem 1.4]{Yalcin:2017}. A $kG$-module $M$ is an \emph{endo-permutation} module if $\End_k (M) \cong M^* \otimes_k M$ is a permutation $kG$-module. For a $G$-set $X$, the kernel of the augmentation map $\varepsilon\colon kX \to k$ is an endo-permutation $kG$-module (see \cite{Alperin:2001,Alperin:2001a}). %\cite{Alperin:2001}
 Endo-permutation modules constructed in this way are called \emph{relative syzygies}. In \cite{Yalcin:2017}, the following theorem is proved: 

\begin{theorem}[{\cite[Theorem 1.2]{Yalcin:2017}}]\label{thm:main} Let $G$ be a finite $p$-group, and $k$ be a field of characteristic $p$. If $X$ is a finite $\un{n}$-Moore $G$-space over $k$, then the reduced homology module $\widetilde H_{n} (X, k)$ in dimension $n=\un{n}(1)$ is an endo-permutation $kG$-module generated by relative syzygies.
\end{theorem}

A version of this theorem also holds for arbitrary finite groups (see \cite[Theorem 7.10]{Gelvin:2021}).  Later these results were extended in the context of endotrivial complexes of $p$-permutation $kG$-modules by Miller \cite{Miller:2024}.  In \cite{Miller:2025a}, Miller
gives a complete classification of endotrivial complexes, and hence determine the Picard group  of the tensor-triangulated category $K^b ({}_{kG} \textbf{triv})$, the bounded homotopy category of $p$-permutation modules, which is the category recently considered by Balmer and Gallauer  in 
\cite{Balmer:2023,Balmer:2023a,Balmer:2025a}.

%%%%%%%%%%%%%%%%%%%%%%%%%%%%
    
%\bibliographystyle{ih_test}                            
%\bibliography{survey}
%\end{document}

%%%%%%%%%%%%%%%%%%%%%%%%%%%%%%
\providecommand{\bysame}{\leavevmode\hbox to3em{\hrulefill}\thinspace}
\providecommand{\MR}{\relax\ifhmode\unskip\space\fi MR }
\providecommand{\MRhref}[2]{%
  \href{http://www.ams.org/mathscinet-getitem?mr=#1}{#2}
}
\providecommand{\href}[2]{#2}

\end{document}